\documentclass{conm-p-l}
\usepackage{amssymb}
\usepackage{hyperref}
\newcommand{\Z}{\mathbb Z}
\newcommand{\Q}{\mathbb Q}
\newcommand{\R}{\mathbb R}
\newcommand{\C}{\mathbb C}
\newcommand{\X}{\mathbb X}
\newcommand{\F}{\mathcal F}
\renewcommand{\O}{\mathcal O}
\newcommand{\p}{\frak p}
\renewcommand{\H}{\mathbb H}
\newcommand{\B}{\mathcal B}
\newcommand{\PSL}{\operatorname{PSL}}
\newcommand{\Isom}{\operatorname{Isom}}
\newcommand{\SL}{\operatorname{SL}}
\newcommand{\GL}{\operatorname{GL}}
\newcommand{\Gal}{\operatorname{Gal}}
\newcommand{\tr}{\operatorname{tr}}

\newcommand{\snap}{{Snap}}
\newcommand{\snappea}{{SnapPea}} 
\newcommand{\Bloch}{\mathcal B}
\newcommand{\Borel}{\operatorname{Borel}}
\newcommand{\Prebloch}{\mathcal P}
\newcommand{\cs}{\operatorname{cs}}
\newcommand{\vol}{\operatorname{vol}}

\newtheorem{theorem}{Theorem}[section]
\newtheorem*{theorem*}{Theorem}
\newtheorem{conjecture}{Conjecture}
\newtheorem{question}[theorem]{Question}
\newtheorem{lemma}[theorem]{Lemma}

\theoremstyle{definition}
\newtheorem{definition}[theorem]{Definition}
\newtheorem{remark}[theorem]{Remark}
\newtheorem{problem}[theorem]{Problem}

\begin{document}
\title{Realizing arithmetic invariants of hyperbolic $3$--manifolds}
\author{Walter D. Neumann}
\address{Barnard College, Columbia University, New York, NY 10027,
  USA}
\email{neumann@math.columbia.edu}
\thanks{This work was supported by
the NSF. Useful conversations and correspondence with Alan Reid and
Christian Zickert are gratefully acknowledged.}
\maketitle
These are notes based on the course of lectures on arithmetic invariants
of hyperbolic manifolds given at the workshop associated with the last
of three ``Volume Conferences,'' held at Columbia University, LSU
Baton Rouge, and Columbia University respectively in March 2006,
May/June 2007, June 2009.

The first part of the lecture series was expository, and since most of
the material is readily available elsewhere, we move rapidly over it
here 
(the very first lecture was a rapid introduction to algebraic number
theory, here compressed to less than 2 pages, but hopefully sufficient
for the topologist who has never had a course in algebraic number
theory). Section \ref{sec:invariants} on arithmetic invariants has
some new material, while Section \ref{sec:realizing} describes a
question that Alan Reid and the author first asked about 20 years ago,
and describes a very tentative approach. It is here promoted to a
conjecture, in part because the author believes he is safe from
contradiction in his lifetime.

In its simplest form the conjecture says:
\begin{conjecture}
Every non-real concrete
number field $k$ and every quaternion algebra over it arise as the
invariant trace field and invariant quaternion algebra of some
hyperbolic manifold.   
\end{conjecture}
With an excess of optimism, one might add to the
conjecture that, moreover, every set of primes of $\O_k$ arises as the
set of primes in denominators in the invariant trace ring of one of
these hyperbolic manifolds.

Section \ref{sec:realizing} describes the already mentioned tentative
first step for a program of proof, which the author has revisited
over many years without significant advance.

We discuss also the question whether the Bloch invariants of
manifolds with a given invariant trace field $k$ generate the Bloch
group $\B(k)$ for that number field, or even whether their extended Bloch
invariants generate $K_3^{\text{\rm ind}}(k)$.

\section{Notation and terminology for algebraic number theory}  
\subsection{Number fields}
A \emph{number field} $K$ is a finite extension of $\Q$. That is, $K$
is a field containing $\Q$, and finite-dimensional as a vector space
over $\Q$. This dimension $d$, denoted $d=[K:\Q]$, is the
\emph{degree} of the number field. $K$ has exactly $d$ embeddings into
the complex numbers,
$$\theta_i\colon K\to \C,\qquad i=1,\dots,d=r_1+2r_2\,,$$
where $r_1$ is the number of them with real image, and the remaining
embeddings come in $r_2$ complex conjugate pairs. 
Indeed, the ``Theorem of the Primitive Element'' implies that $K$ is
generated over $\Q$ by a single element, from which if follows that
$K\cong\Q[x]/(f(x))$ with $f(x)$ an irreducible polynomial of degree
$d$; the embeddings $K\to\C$ arise by mapping the generator $x$ of
$K$ to each of the $d$ zeros in $\C$ of $f(x)$.

A \emph{concrete number field} is a number field $K$ with a chosen
embedding into $\C$, i.e., $K$ given as a subfield of $\C$.  The union
of all concrete number fields is the field of algebraic numbers in
$\C$, which is the concrete algebraic closure $\overline \Q\subset\C$
of $\Q$.

An \emph{algebraic integer} is a zero of a monic polynomial with
rational integer coefficients. The algebraic integers in $K$ form a
subring $\O_K\subset K$, the \emph{ring of integers of $K$}. It is a
Dedekind domain, which is to say that any ideal in $\O_k$ factors
uniquely as a product of prime ideals.  Each prime ideal $\p$ 
(or ``prime'' for short) 
of $\O_K$ is a divisor of a unique ideal $(p)$ with $p\in \Z$ a
rational prime (determined by $|\O_K/\p|=p^e$ for some $e>0$). The
factorization of $(p)$ as a product $(p)=\p_1^{f_1}\dots\p_k^{f_k}$ of
primes of $\O_K$ follows patterns which can be found in any text on
algebraic number theory. In particular, the exponents $f_i$ are $1$
for all but a finite number of primes $\p$ of $\O_K$, which are called
\emph{ramified}.

For the ring of integers $\O_\Q=\Z$ of $\Q$, every ideal is principal,
and the factorization of the ideal $(n)$ into a product of ideals
$(p_i)$ expresses the familiar unique prime factorization of rational
integers. In general $\O_K$ is a unique factorization domain (UFD) if
and only if it is a PID (every ideal is principal), which is somewhat
rare. It is presumed to happen infinitely often, but this is not
proven.

Given a prime $\p$ of $\O_k$, there is a multiplicative norm
$||.||_\p$ defined for $a\in\O_K$ by $||a||_\p:=c^{-r}$, where $\p^r$
is the largest power of $\p$ which ``divides'' $a$ (i.e., contains
$a$) and $c>1$ is some constant\footnote{The value of $c$ is
  unimportant for topological considerations but is standardly taken
  as $c=N(\p):=|\O_K/\p|$}; the norm is then determined for arbitrary
elements of $K$ by the multiplicative property
$||ab||_\p=||a||_\p||b||_\p$. This norm determines a translation
invariant topology on $K$ and the completion of $K$ in this topology
is a field denoted $K_\p$. The unit ball around $0$ in $K_\p$ is its \emph{ring
  of integers $\O_{K_\p}$}, and the open unit ball is the unique
maximal ideal in this ring.  The norm $||.||_\p$ is
\emph{non-Archimedean}, i.e., it satisfies the strong triangle
inequality $||a+b||_\p\le \max(||a||_\p,||b||_\p)$.  Up to equivalence
(norms are \emph{equivalent} if one is a positive power
of the other), the only non-Archimedean multiplicative norms are the
ones just described, and the only other multiplicative norms on
$K$ are the norms $||a||_\theta:=|\theta(a)|$ given by absolute value
in $\C$ for an embedding $\theta\colon K\to \C$. The completion of $K$
in the topology induced by one of these is $\R$ or $\C$ according as
the image of $\theta$ lies in $\R$ or not.

The fields $\R$, $\C$, $K_\p$ arising from completions are \emph{local
  fields}\footnote{The definition of local field is: non-discrete
  locally compact topological field. The ones mentioned here are all
  that exist in characteristic $0$.}. The name is geometrically
motivated: one thinks of $\O_K$ as a ring of functions on a ``space''
with a ``finite point'' for each prime ideal, plus $r_1+r_2$
``infinite points'' corresponding to the embeddings in $\R$ and $\C$;
``local'' means focusing on an individual point. One therefore refers
to an embedding of $K$ into $K_p$ as a ``finite place'' and an
embedding into $\R$ or $\C$ as an ``infinite place,'' and if an object
$A$ associated with $K$ (e.g., an algebra $A$ over $K$) has
corresponding objects associated to each place (e.g., $A\otimes K_\p$,
$A\otimes\R$, $A\otimes\C$) then a ``property of $A$ at the (finite or
infinite) place'' means that property for the associated object. We
stress that an ``infinite place'' refers to the embedding of $K$ in
$\C$ up to conjugation (even though conjugate embeddings may have
different images), so there are $r_1$ real places and just $r_2$
complex places.

\subsection{Quaternion algebras}
References for this section are \cite{vigneras} and \cite{cghn}. A
\emph{quaternion algebra} over a field $K$ is a simple algebra over
$K$ of dimension $4$ and with center $K$. The simplest example is the
algebra $M_2(K)$ of $2\times2$ matrices over $K$. This is the only
quaternion algebra up to isomorphism for $K=\C$. For $K=\R$ there are exactly two,
namely $M_2(\R)$ and the Hamiltonian quaternions. The situation for
the non-Archimedean local fields $K_\p$ is similar: there are exactly
two quaternion algebras over each of them, one being the trivial one
$M_2(K_\p)$ and the other being a division algebra.  In each case the
trivial quaternion algebra $M_2$ is called \emph{unramified} and the
division algebra is called \emph{ramified}. For a number field $K$ the
classification of quaternion algebras over $K$ is as follows:
\begin{theorem}[Classification]
  A quaternion algebra $E$ over $K$ is ramified at only finitely many
  places (i.e., only finitely many of the $E\otimes K_\p$ and
  $E\otimes\R$'s are division algebras) and is determined up to
  isomorphism by the set of these ``ramified places.'' The number of
  ramified places is always even, and every set of places of $K$ of
  even size arises as the set of ramified places of a quaternion
  algebra over $K$.
\end{theorem}
A quaternion algebra $E$ over $K$ can always be given in terms of
generators and relations in the form
$$E=K\langle {i}, {j}:{i}^2=\alpha,~{j}^2=\beta,~ij=-ji\rangle\,,$$
with $\alpha,\beta\in K^*$. The \emph{Hilbert symbol} notation
$\bigl\{\!\frac{\alpha,\beta}K\!\bigr\}$ refers to this quaternion
algebra. For example, $\bigl\{\!\frac{-1,-1}\R\!\bigr\}$ is Hamilton's
quaternions, and $\bigl\{\!\frac{1,\beta}K\!\bigr\}=M_2(K)$ for any
$K$. The Hilbert symbol for a given quaternion algebra is far from
unique, but computing the ramification---and hence the isomorphism
class---of a quaternion algebra from the Hilbert symbol is not hard,
and is described in \cite{vigneras}, see also \cite{cghn} for a
description tailored to 3-manifold invariants.

In terms of the above presentation, the map $i\mapsto -i, j\mapsto -j,
ij\mapsto -ij$ of a quaternion algebra $E$ to itself is an
anti-automorphism called conjugation, and the \emph{norm} of
$x=a+ib+jc+ijd\in E$ is defined as $N(x):=x\bar x=a^2+\alpha b^2+\beta
c^2+\alpha\beta d^2$.

\subsection{Arithmetic subgroups of $\SL(2,\C)$ and $\PSL(2,\C)$}
\label{subsec:arith}
For a quaternion algebra $E$ over $K$ 
the set $\O_E$ of integers of $E$ (elements
which are zeros of monic polynomials with coefficients in $\O_K$) does
not form a subring. One considers instead an \emph{order}
in $E$: any subring $\O$ of $E$, contained in $\O_E$ and
containing $\O_K$ and of rank 4 over $\O_K$. $E$ has infinitely many
orders; we just pick one of them.

The subset $\O^1\subset \O$ of elements of norm 1 is a subgroup.  At
any complex place, $E$ becomes $E\otimes\C=M_2(\C)$ and $\O^1$ becomes
a subgroup of $\SL(2,\C)$, while at an unramified real place $E$ becomes
$E\otimes\R=M_2(\R)$ and $\O^1$ becomes a subgroup of $\SL(2,\R)$. We
thus get an embedding of $\Gamma:=\O^1/\{\pm1\}$
$$\Gamma\subset 
\prod_{i=1}^{r_2}\PSL(2,\C)
\times
\prod_{j=1}^{r_1^u}\PSL(2,\R)\,,
$$ 
where $r_1^u$ is the number of unramified real places of $K$.  This
subgroup is a lattice (discrete and of finite covolume).

If $r_1^u=0$ and $r_2=1$ this gives an arithmetic subgroup of
$\PSL(2,\C)$, and similarly for $r_1^u=1, r_2=0$ and $\PSL(2,\R)$. Up
to commensurability this group only depends on $E$ and not on the
choice of order $\O$. Any subgroup commensurable with an arithmetic
subgroup---i.e., sharing a finite index subgroup with it up to
conjugation---is, by definition, also arithmetic.

The general definition of an arithmetic group is in terms of the set
of $\Z$-points of an algebraic group which is defined over $\Q$.  But Borel
shows in \cite{borel} that all arithmetic subgroups of $\PSL(2,\C)$
(and $\PSL(2,\R)$) can be obtained as above.

Arithmetic orbifolds (orbifolds $\H^3/\Gamma$ with $\Gamma$
arithmetic) are very rare---there are
only finitely many of bounded volume---but surprisingly common among
the manifolds and orbifolds of smallest volume.
\section{Arithmetic invariants of hyperbolic manifolds}\label{sec:invariants}

\subsection{Invariant trace field and quaternion algebra}
A \emph{Kleinian group} $\Gamma$ is a discrete subgroup of
$\PSL(2,\C)=\Isom^+(\H^3)$ for which $M=\H^3/\Gamma$ is finite volume
($M$ may be an orbifold).
Let $\overline {\Gamma} \subset \SL(2,\C )$ be the inverse image of
$\Gamma$ under the projection $\SL(2,\C)\to \PSL(2,\C)$.

\begin{definition}
  The \emph{trace field} of $\Gamma$ (or of $M=\H^3/\Gamma$) is the
  field $\tr(\Gamma)$ generated by all traces of elements of
  $\overline \Gamma$. We also write $\tr(M)$.

  The \emph{invariant trace field} is the field
  $k(\Gamma):=\tr(\Gamma^{(2)})$ where ${\Gamma}^{(2)}$ is the group
  generated by squares of elements of ${\Gamma}$. It can also be
  computed as $k(\Gamma )=\Q (\{ (\tr(\gamma))^2\mid\gamma\in\overline
  {\Gamma }\} )$ (\cite{reid}, see also \cite{nr}). We also write 
  $k(M)$.

The \emph{invariant quaternion algebra} is the $k(\Gamma
)$-subalgebra of $M_2(\C )$ ($2\times 2$ matrices over $\C$)
generated over $k(\Gamma )$ by the elements of $\overline \Gamma^{(2)}$.
It is denoted  $A(\Gamma)$ or $A(M)$.
\end{definition}
\begin{theorem} $k(\Gamma )$ and $A(\Gamma)$ are
  commensurability invariants of $\Gamma$.\par
\end{theorem}

If $\Gamma$ is arithmetic, then $k(\Gamma )$ and $A(\Gamma )$ equal
the defining field and defining quaternion algebra of $\Gamma$, so
they form a complete commensurability invariant. They are not a
complete commensurability invariant in the non-arithmetic case.  

An obvious necessary condition for arithmeticity is that $k(\Gamma )$
have only one non-real complex embedding (it always has at least one).
Necessary and sufficient is that in addition all traces should be
algebraic integers and $A(\Gamma)$ should be ramified at all real
places of $k$. See \cite{reid}. Equivalently, each $\gamma\in
\overline {\Gamma}$ ${\rm trace}(\gamma^2)$ should be an algebraic
integer whose absolute value at all real embeddings of $k$ is bounded
by $2$. 

These invariants are already quite powerful invariants of a hyperbolic
manifold. For example, if a hyperbolic manifold $M$ is commensurable
with an amphichiral manifold $N$ (i.e., $N$ has an orientation
reversing self-homeomorphism) then $k(M)=\overline{k(M)}$ and
$A(M)=\overline{A(M)}$ (complex conjugation).

If $M$ has cusps then the invariant quaternion algebra is always
unramified, so it gives no more information than the invariant trace
field, but for closed $M$ unramified invariant quaternion algebras are
uncommon; for example among the almost $40$ manifolds in the Snappea
closed census \cite{snappea} which have invariant trace field
$\Q(\sqrt{-1})$, only two have unramified quaternion algebra.

\subsection{The PSL-fundamental class}

For details on what we discuss here see \cite{extbloch,
neumann-yang2, neumann-yang3} or the expository article
\cite{neumann-hilbert}. 

\subsubsection{PSL-fundamental class of a hyperbolic manifold} The
\emph{PSL-fundamental class} of 
$M$ is a
homology class 
$$[M]_{PSL}\in H_3(\PSL(2,\C)^\delta;\Z)\,,$$
where the superscript $\delta$ means ``with discrete topology''.

This class is easily described if $M$ is compact. Write $M=\H^3/\Gamma$ with
$\Gamma\subset\PSL(2,\C)$. The PSL-fundamental class is the image of
the fundamental class of $M$ under the map
$H_3(M;\Z)=H_3(\Gamma;Z)\to H_3(\PSL(2,\C)^\delta;\Z)$, where the
first equality is because  $M$ is a $K(\Gamma,1)$-space.
If $M$ has cusps one obtains first a class in
$H_3(\PSL(2,\C)^\delta,P;\Z)$, where $P$ is a maximal parabolic
subgroup of $\PSL(2,\C)^\delta$. One then uses a natural splitting
of the map $$H_3(\PSL(2,\C)^\delta;\Z)\to H_3(\PSL(2,\C)^\delta,P;\Z)$$
to get $[M]_{\PSL}$. This was described in \cite{extbloch}
and proved carefully by Zickert in \cite{zickert1}, who shows that the
class in $H_3(\PSL(2,\C)^\delta,P;\Z)$ depends on choices of horoballs
at the cusps, but the image $[M]_{\PSL}\in H_3(\PSL(2,\C)^\delta;\Z)$
does not.

The group $\Gamma\subset \PSL(2,\C)$ can be conjugated to lie in
$\PSL(2,K)$ for a number field $K$ (which can always be chosen to be a
quadratic extension of the trace field, but there is generally no
canonical choice), so the PSL-fundamental class is then defined in
$H_3(\PSL(2,K);\Z)$.

The following theorem, which holds also with $\PSL$ replaced by $\SL$,
summarizes results of various people, see \cite{neumann-yang3} and
\cite{zickert2} for more details.

\begin{theorem}\label{various results}
$H_3(\PSL(2,\C);\Z)$ is the direct sum of
its torsion subgroup, isomorphic to $\Q/\Z$, and an
infinite dimensional $\Q$ vector space.

If $k\subset\C$ is a number field then
$H_3(\PSL(2,k);\Z)$ is the direct sum of 
its torsion subgroup and
$\Z^{r_2}$, where $r_2$ is the number of conjugate pairs of
complex embeddings of $k$.
Moreover, the map $H_3(\PSL(2,k);\Z)\to H_3(\PSL(2,\C);\Z)$
has torsion kernel.
\end{theorem}
The \emph{Rigidity Conjecture}, which is about 30 years old (see
\cite{neumann-hilbert} for a discussion), posits
that each of the following equivalent statements is true:

\begin{conjecture}
{\rm(1)}~~$H_3(\PSL(2,\C)^\delta;\Z)$ is countable. \\
{\rm(2)}~~$H_3(\PSL(2,\overline{\Q})^\delta;\Z)=H_3(\PSL(2,\C)^\delta;\Z)$\\
{\rm(3)}~~$H_3(\PSL(2,\C)^\delta;\Z)$ is the union of the images
  of the maps $H_3(\PSL(2,K);\Z)\to H_3(\PSL(2,\C)^\delta;\Z)$, as $K$
  runs through all concrete number fields.
\end{conjecture}

\subsection{Invariants of the PSL-fundamental class} There is a
homomorphism $$\hat c\colon H_3(\PSL(2,\C);\Z)\to\C/\pi^2\Z$$ called
the ``Cheeger-Simons class'' (\cite{cheeger-simons}) whose real and
imaginary parts give Chern-Simons invariant and volume: $$\hat
c([M]_{PSL})= cs(M)+i\vol(M)\,.$$ The Chern-Simons invariant
here is the Chern-Simons invariant of the flat connection, which is
defined for any complete hyperbolic manifold $M$ of finite volume. If
$M$ is closed the Riemannian Chern-Simons invariant
$\operatorname{CS}(M)\in \R/2\pi^2$ is also defined; it reduces to
$\cs(M)$ mod $\pi^2$. See \cite{extbloch} for details.

We denote the
homomorphisms given in the obvious way by the real and imaginary parts
of $\hat c$ by:
$$\cs\colon H_3(\PSL(2,\C);\Z)\to \R/\pi^2\Z\,,\qquad
\vol\colon H_3(\PSL(2,\C);\Z)\to \R\,.$$ The homomorphism $\cs$ is
injective on the torsion subgroup of $H_3(\PSL(2,\C);\Z)$. A standard
conjecture that appears in many guises in the literature (see
\cite{neumann-hilbert} for a discussion) is:
\begin{conjecture}
\label{rama} 
  The Cheeger-Simons class is injective.  That is, volume and
  Chern-Simons invariant determine elements of $H_3(\PSL(2,\C);\Z)$
  completely. 
\end{conjecture}
If $k$ is an algebraic number field and
$\sigma_1,\dots,\sigma_{r_2}\colon k\to\C$ are its different complex
embeddings up to conjugation then denote by $\vol_j$ the composition
$$\vol_j=\vol\circ(\sigma_j)_*\colon H_3(\PSL(2,k);\Z)\to\R.$$ The map
$$\Borel:=(\vol_1,\dots,\vol_{r_2})\colon
H_3(\PSL(2,k);\Z)\to\R^{r_2}$$ is called the \emph{Borel regulator}.
\begin{theorem}\label{borel theorem} The Borel regulator maps
  $H_3(\PSL(2,k);\Z)/\text{\rm Torsion}$ injectively onto a full
  sublattice of $\R^{r_2}$. 
\end{theorem}

By Theorem \ref{various results} and the discussion above,
$\cs(M)\in\R/\Z$ and $\Borel([M]_{PSL}) \in \R^{r_2(k)}$ determine the
PSL-fundamental class $[M]_{PSL}\in H_3(\PSL(2,\C);\Z)$ completely.

These invariants are computed by the program \snap\ (see
\cite{cghn}). Snap does this via a more easily computed invariant
which we describe next.

\subsection{Bloch group and Bloch invariant} 
The \emph{Bloch group $\Bloch(\C)$} is the quotient of
$H_3(\PSL(2,\C)^\delta;\Z)$ by its torsion subgroup. It has the
advantage that it has a simple symbolic description and the image of
$[M]_{\PSL}$ in $\Bloch(\C)$ is readily computed from an ideal
triangulation.

The Bloch group is defined for any field. There are different
definitions of it in the literature; they differ at most
by torsion and agree with each other for algebraically closed fields
(see, e.g., \cite{dupont-sah}). We use the following.

\begin{definition}\label{def-bloch} Let $K$ be a field. The {\em
pre-Bloch group $\Prebloch(K)$} is the quotient of the free
$\Z$-module $\Z (K-\{0,1\})$ by all instances of the following
relation:
$$ [x]-[y]+[\frac
yx]-[\frac{1-x^{-1}}{1-y^{-1}}]+[\frac{1-x}{1-y}]=0, $$
called the \emph{five term relation}. The \emph{Bloch group
  $\Bloch(K)$} is the kernel of the map
\begin{equation}
  \label{eq:lambda}
  \Prebloch(k)\to
K^*\wedge_\Z K^*,\quad [z]\mapsto 2(z\wedge(1-z))\,. 
\end{equation}
\end{definition}



Suppose we have an ideal triangulation of a hyperbolic $3$-manifold
$M$ using ideal hyperbolic simplices with cross ratio parameters
$z_1,\dots,z_n$. This ideal triangulation can be a genuine ideal
triangulation of a cusped $3$-manifold, or a deformation of such a one
as used by \snap\ and \snappea\ to study Dehn filled manifolds, but it
may be more generally any ``degree one
triangulation''; see \cite{neumann-yang3}.

\begin{definition} The {\em Bloch invariant $\beta(M)$} is the element
  $\sum_1^n \pm[z_j]\in\Prebloch(\C)$ with signs as explained
  below. It lies in $\Bloch(\C)$ by \cite{neumann-yang3}.
\end{definition}

The cross-ratio parameter of an ideal simplex depends on a chosen
ordering of the vertices, and the sign in the above sum reflects
whether or not this ordering orients the
simplex as it is oriented as part of the degree one
triangulation. 

If the $z_j$'s all belong to a subfield $K\subset \C$, we may consider
$\beta(M)$ as an element of $\Bloch(K)$. But it is then necessary to
assume that the vertex orderings of the simplices match on common faces.
If not, then
$\sum_1^n \pm[z_j]$ may differ from $\beta(M)$ by a torsion element (of
order dividing $12$; this torsion issue does not arise in
$\Bloch(\C)$, which is torsion-free).  Not every triangulation has
compatible vertex-orderings for the simplices, although a
triangulation can always be refined to one which does.
\begin{theorem}\label{th:in k} If $M$ has cusps then 
  $\beta(M)$ is actually defined in $\Bloch(k)$, for the invariant
  trace field $k$, while if $M$ is closed this holds for $2\beta(M)$.
\end{theorem}
This was known in the cusped case (the simplex parameters of an ideal
triangulation then lie in $k$, see \cite{nr}) but it was only known up
to a higher power of $2$ in the closed case (\cite{neumann-hilbert,
  neumann-yang3}). The proof, joint with Zickert (but mostly Zickert),
is at the end of this subsection.

Since $\beta(M)$ only loses torsion information over $[M]_{\PSL}$, the
Borel regulator $\Borel(M)$ can be computed from $\beta(M)$. It is
computed from the simplex parameters $z_i$ as follows. The $z_i$
generate a field $K$ which contains the invariant trace field $k$ of
$M$. The $j$-th component $\vol_j([M]_{PSL})$ of $\Borel(M)$
is $$\Borel(M)_j=\sum_{i=1}^n \pm D_2(\tau_j(z_i)),$$ where
$\tau_j\colon K\to\C$ is any complex embedding which extends
$\sigma_j\colon k\to \C$. Here the signs are as above, and $D_2$ is
the ``Wigner dilogarithm function'' $$D_2(z) = \operatorname{Im}
\ln_2(z) + \log |z|\arg(1-z),\quad z\in \C -\{0,1\},$$ where
$\ln_2(z)$ is the classical dilogarithm function. $D_2(z)$ can also be
defined as the volume of the ideal simplex with parameter $z$.

Recall that $k$ is a concrete number field, i.e., it comes as a
subfield of $\C$. The component of $\Borel(M)$ corresponding to this
embedding in $\C$ is $\pm\vol(M)$, and it has maximal absolute value
among the components of $\Borel(M)$ (see \cite{neumann-yang3}).  This
restricts which elements of $\Bloch(k)$ can be the Bloch invariant of
a hyperbolic 3-manifold. A related (and conjecturally equivalent)
restriction is in terms of the Gromov norm, which is defined on
$\Bloch(k)$ (see \cite{neumann-yang3}); the Bloch invariants of
hyperbolic manifolds are constrained to lie in the cone over a single
face of the norm ball.

Nevertheless, it is plausible that the Bloch group can be generated by
Bloch invariants of $3$-manifolds. No obstructions to this are known,
and there is (very mild) experimental evidence for it for low degree
fields which appear as invariant trace fields of manifolds in the
cusped and Snappea closed censuses \cite{census, snappea}; some
computations related to this are in \cite{cghn}.  So we ask:
\begin{question}
  Is $\Bloch(k)$ generated by Bloch invariants of hyperbolic manifolds
  with invariant trace field in $k$; how about $\Bloch(k)\otimes\Q$
  over $\Q$? 
\end{question}
\begin{proof}[Proof of Theorem \ref{th:in k}] (See also \cite{zickert3}.)
  Since the theorem is known in the cusped case we assume $M$ is
  closed. We will need Suslin's version of the Bloch group
  \cite{suslin}, defined by omitting the factor 2 in the map 
  \eqref{eq:lambda} in the definition above. We will denote it
  $\Bloch_S(K)$. Clearly, $\Bloch_S(K)\subset \Bloch(K)$, and the
  quotient $\Bloch(K)/\Bloch_S(K)$ is of exponent 2 (one can show it
  is infinitely generated if $K$ is a number field). We will actually
  show that $\beta(M)\in \Bloch_S(k)$.

  We will use Suslin's theorem that $\Bloch_S(K)$ is a quotient of
  $K_3^{\text{\rm ind}}(K)$ by a finite cyclic group (\cite{suslin},
  see also \cite{zickert3}).

  The geometric $\PSL$--representation of $\pi_1(M)$ lifts to a
  representation $\pi_1(M)\to \SL(2,\C)$. The set of such lifts is in
  one-one correspondence with spin structures on $M$; we just pick one
  for now. The image $\Gamma$ of such a lifted representation lies in
  the quaternion algebra $\Q \Gamma$, which can be unramified by
  extending its scalars to a quadratic extension field $K'$ of the
  trace field $K$ (such a $K'$ can be taken as $K(\lambda)$ for any
  eigenvalue $\lambda$ of a nontrivial element of $\Q\Gamma$; see
  e.g., \cite{mr}). We get $\pi_1(M)\hookrightarrow \GL(2,K')$,
  leading to a ``GL--fundamental class'' $[M]_{\GL}\in
  H_3(GL(2,K');\Z)$. There is a natural map $H_3(\GL(2,K');\Z)\to
  K_3^{\text{\rm ind}}(K')$ (see, e.g., \cite{zickert3}), and we
  denote the image of $[M]_{\GL}$ by $[M]_K\in K_3^{\text{\rm
      ind}}(K')$. Now the non-trivial element of $\Gal(K'/K)$
  preserves traces of $\pi_1(M)\to \GL(2,\C)$, so it takes this
  representation to a representation which is equivalent over $\C$. It
  therefore fixes the Borel invariant and Chern-Simons invariant of
  $[M]_K$. The Chern-Simons invariant on $K_3^{\text{\rm ind}}$ takes
  values in $\C/4\pi^2\Z$, and this Chern-Simons invariant and Borel
  invariant together determine any element of $K_3^{\text{\rm
      ind}}(K')$ (see \cite{zickert3}). Thus the class $[M]_K$ is
  invariant under $\Gal(K'/K)$, and since $K_3^{\text{\rm ind}}$
  satisfies Galois descent (Merkuriev and Suslin \cite{m-suslin}),
  this class lies in $K_3^{\text{\rm ind}}(K)$.  We note, however,
  that a priori $[M]_K$ may depend on which lift
  $\pi_1(M)\hookrightarrow \SL(2,\C)$ we started with.  

  By \cite[Theorem 2.2]{nr} the trace field $K$ is a multi-quadratic
  extension of the invariant trace field $k$, with Galois group
  $\Gal(K/k)\cong (\Z/2)^r$ for some $r$. This group permutes the
  lifts $\pi_1(M)\hookrightarrow \SL(2,\C)$ and hence acts on the
  elements $[M]_K\in K_3^{\text{\rm ind}}(K)$ defined by these
  lifts. In  \cite{pitsch} it is shown that $[M]_K$ is changed by at
  most the unique element of order 2 (and such a change can
  occur). Thus $2[M]_K$ is invariant under this Galois group, and is hence an
  invariant of $M$ in $K_3^{\text{\rm ind}}(k)$ which is independent
  of the lift to $SL(2,\C)$.

  The Bloch invariant $2\beta(M)$ is the image of
  $2[M]_K$ under a natural transformation from $K_3^{\text{\rm ind}}$
  to $\Bloch_S$, so the theorem is proved.
 \end{proof}
\begin{problem}
  Find a general explicit way to obtain a representative for 
  $\beta(M)$ in $\Bloch_S(k)$ if $M$ is closed.
\end{problem}
\begin{remark}
  In \cite{zickert3} Zickert points out that $\beta(M)\in \Bloch(k)$
  may not be in its subgroup $\Bloch_S(k)$ if $M$ has cusps, in
  contrast to the closed case. An example is the manifold $M009$ in
  the census \cite{census}.
\end{remark}

\subsection{Extended Bloch group}
The \emph{extended Bloch group} $\widehat\Bloch(\C)$ of \cite{extbloch} is
defined by replacing $\C-\{0,1\}$ by a $\Z\times\Z$--cover in the
definition of Bloch group, and appropriately lifting the 5-term
relation and the map $\lambda$.  There are two different versions of
this defined in \cite{extbloch}; we will write $\widehat\Bloch_{\PSL}(\C)$ for
the first and $\widehat\Bloch(\C)$ for the second (they were denoted
$\widehat\Bloch(\C)$ and $\mathcal E\Bloch(\C)$ respectively in
\cite{extbloch}).
\begin{theorem}
  There are natural isomorphisms
  $H_3(\PSL(2,\C)^\delta;\Z)\cong\widehat\Bloch_{\PSL}(\C)$ and
  $H_3(\SL(2,\C)^\delta;\Z) \cong \widehat\Bloch(\C)\cong K_3^{\text{\rm ind}}(\C)$.  (See
  \cite{extbloch} and \cite{goette-zickert} respectively.)
\end{theorem}

The program \snap\ actually computes the element in
$\widehat\Bloch_{\PSL}(\C)$, and then prints the Borel regulator and
Chern-Simons invariant, which, as already mentioned, determine this
element. In \cite{zickert2} Zickert gives a much simpler way of
computing the element of $\widehat\Bloch_{\PSL}(\C)$ than the one
currently used by \snap.  This has now been implemented in SnapPy
\cite{SnapPy} by Matthias Goerner. 

In \cite{zickert3} Zickert extends the definition of extended Bloch
groups to number fields and shows
\begin{theorem}
There is a natural
isomorphism $\widehat\Bloch(K)\cong K_3^{\text{\rm ind}}(K)$ for any number
field $K$.
\end{theorem}
Moreover, as mentioned in the proof of Theorem \ref{th:in k}, for a
closed hyperbolic manifold with spin structure Zickert shows that
there is a natural invariant $[M]_K\in \widehat\B(k)=K_3^{\rm ind}(k)$
which lifts the Bloch invariant. If $M$ has cusps he shows that
$[M]_K$ is defined in $\widehat \Bloch(K)\otimes\Z[\frac12]$, where
$K$ is the trace field, and his arguments show that $8[M]_K$ is well
defined in $\widehat\Bloch(k)=K_3^{\rm ind}(k)$ (it is easily seen
that $4[M]_K$ is well defined in $K_3^{\rm ind}(\C)$).
\subsection{Mutation} 
If a hyperbolic manifold $M$ contains an essential $4$--punctured
sphere then one can cut along this embedded surface and re-glue by an
involution, and it is well known that this process, called
\emph{Conway mutation} or simply \emph{mutation}, yields a hyperbolic
manifold which shares many properties with $M$, for example it has the
same volume and Chern-Simons invariant \cite{ruberman}, and if $M$ was
a knot complement, many knot theoretic invariants are preserved
too. It is folk knowledge that the Bloch invariant is preserved (and
hence also $[M]_{\PSL}$) but there is no proof in the literature, so
we give a direct proof here. There are other types of mutation,
sometimes called ``generalized mutation.'' For example one can mutate
along any essential embedded $3$--punctured sphere.
\begin{theorem}
  If $M$ and $M'$ are hyperbolic manifolds related by Conway mutation,
  then $[M]_{\PSL}=[M']_{\PSL}$. 

If they are related by mutation on a
$3$--punctured sphere then $[M]_{\PSL}$ and $[M']_{\PSL}$ differ by
  the element of order 2 in $H_3(\PSL(2,\C)^\delta;\Z)$. 

For  any generalized mutation $[M]_{\PSL}$ and $[M']_{\PSL}$ differ by an
  element of finite order.  Every torsion element of
  $H_3(\PSL(2,\C)^\delta;\Z)$ arises this way.
\end{theorem}
\begin{proof}
  Suppose we have an essential embedded two-sided surface
  $\Sigma\subset M$ which has a tubular neighborhood $N$ which admits
  a (necessarily finite order) isometry $\Phi$ which preserves
  orientation and sides of $\Sigma$.
  Denoting by $\phi\colon\Sigma\to \Sigma$ the restriction of $\Phi$,
  let $M'$ be the manifold obtained from $M$ by cutting along $\Sigma$
  and re-gluing by $\phi$. This is ``generalized mutation.'' If
  $\Sigma$ is a $3$-- or $4$--punctured sphere or a closed surface of
  genus $2$ then it can always be positioned to have a $\Z/2$--symmetry.

  Let $K$ be a $K(\PSL(2,\C)^\delta,1)$--space, so
  $H_3(K;\Z)=H_3(\PSL(2,\C)^\delta;\Z)$. The holonomy map
  $\alpha\colon\pi_1(M)\to \PSL(2,\C)$ induces a map
  $\alpha_\sharp\colon M\to K$ (well defined up to homotopy), and
  similarly we have $\alpha'_\sharp\colon M'\to K$. These maps can be
  chosen to agree outside the tubular neighborhood of
  $N\cong\Sigma\times[0,1]$ of $\Sigma$. They then give a map $N\cup
  -N\to K$. Thinking of $N\cup -N$ as
  $\Sigma\times[0,1]\cup-\Sigma\times[0,1]$, it is glued on one end by
  the identity and on the other end by $\phi$, so it is simply the
  mapping torus $T_\phi\Sigma$ of $\phi\colon\Sigma\to\Sigma$. Its
  fundamental group is the semidirect product $\pi_1(\Sigma)\rtimes
  \Z$ and it is represented into $\PSL(2,\C)$ by the homomorphism which
  on $\pi_1(\Sigma)$ is the restriction of $\alpha$ and on a generator
  of $\Z$ is an isometry of $\H^3$ which restricts to a lift to
  $\tilde N$ of the isometry $\Phi\colon N\to N$.

  The images in $H_3(K;\Z)=H_3(\PSL(2,\C)^\delta;\Z)$ of the
  fundamental classes of $M$, $M'$ and $T_\phi\Sigma$ clearly satisfy
  $[M]_{\PSL}-[M']_{\PSL}=[T_\phi\Sigma]_{\PSL}$.  Since
  $T_\phi\Sigma$ is $n$--fold covered by $\Sigma\times S^1$, where $n$
  is the order of $\phi$, the element $[T_\phi\Sigma]_{\PSL}$ is
  $n$--torsion, and is hence determined by the Chern-Simons
  invariant. It describes the change of $[M]_{\PSL}$ under the
  corresponding mutation. Moreover, it depends only on the finite
  order map
  $\Phi\colon N\to N$ and not on the geometry in a neighborhood of
  $N$, since if one deforms the geometry in an equivariant fashion
  then $cs([T_\Phi\Sigma])$ is a continuously varying $n$--torsion
  element in $\R/\pi^2\Z$, hence constant.  So to compute it one just needs to
  compute the change in Chern-Simons invariant for a single
  example. The following two examples thus complete the proof of the
  first two sentences of the theorem.

The Conway and Kinoshita-Teresaka knots, which are
  related by Conway mutation, both have Chern-Simons invariant
  $7.1925796077528967037240463\dots$ (mod $\pi^2$), while the two
  orientations of the Whitehead link are related by a three-punctured
  mutation and have Chern-Simons invariants $\pm \pi^2/4$.

  The final sentence of the theorem is by section 3 of
  \cite{meyerhoff-ruberman}, in which Meyerhoff and Ruberman give
  examples to show any change of Chern-Simons invariant by a rational
  multiple of $\pi^2$ arises by generalized mutation.
\end{proof}
The Riemannian Chern-Simons invariant, defined for a closed hyperbolic
manifold $M$, is a lift of $\cs(M)$ to an invariant defined modulo
$2\pi^2$. It is not uncommon to see claims in the literature that the
Chern-Simons invariant of a cusped manifold can be defined modulo
$2\pi^2$, but we have the following consequence of the above theorem:
\begin{theorem}
  There is no consistent definition of $\cs(M)$ which is well defined
  modulo $2\pi^2$ for cusped manifolds. 
\end{theorem}
\begin{proof}
  Mutation along a thrice-punctured sphere is an involution which
  changes $\cs$ by $\pi^2/2$. Such a change cannot lift to a an order
  two change modulo $2\pi^2$.  
\end{proof}
\subsection{Scissors Congruence}

\def\sc{\operatorname{\mathcal P}}
\def\dehnker{\operatorname{\mathcal D}}

Two hyperbolic manifolds $M_1$ and $M_2$ are \emph{scissors congruent}
if $M_1$ can be cut into finitely many (possibly partially ideal)
polyhedra which can be reassembled to form $M_2$. They are
\emph{stably scissors congruent} if there is some polyhedron $Q$ such
that $M_1+Q$ is scissors congruent to $M_2+Q$ (disjoint union). If
$M_1$ and $M_2$ are either both compact or both non-compact then
stable scissors congruence implies scissors congruence.  The following
follows easily from \cite{neumann-yang1} (see Theorem 7.2 of
\cite{cghn}):
\begin{theorem}\label{sc} Let $K$ be a field that contains the
  invariant trace fields of $M_1$ and $M_2$. Then
$M_1$ and $M_2$ are stably scissors congruent if and only if
 $\Borel(M_1)-\Borel(M_2)$ is the Borel
regulator of an element of $\Bloch(K\cap \R)$
\end{theorem}
\noindent In particular, if $K\cap \R$ is totally real (as is
``usually'' the case) then scissors congruence
class of $M$ is not only determined by $\Borel(M)$ but also determines
it.



\medskip As discussed in \cite{neumann-hilbert}, the following
conjecture would be a consequence of Conjecture \ref{rama}.
\begin{conjecture} The stable scissors congruence class of $M$ is
  determined by $\vol(M)$. \end{conjecture}

In view of the above theorem this conjecture is amenable to experimentation with
\snap; all the evidence from this is positive.

\section{Realizing invariants}\label{sec:realizing}

We know of no way to generate examples of manifolds with given Bloch
invariant. The program \snap\ enables extensive
experimentation---basically casting a fishing line in an ocean of
examples---but even for quadratic number fields, only a few fields
allow manifolds of small enough volume that they can easily
be caught this way.

Moreover, to try to realize such fine arithmetic invariants, one must
first realize any non-real number field as an invariant trace
field. Here we address this question, and that of realizing a
quaternion algebra.  The only general result known in this direction
is a result first observed by Reid and the author, described in
\cite{mr}, that any non-real multi-quadratic extension of $\Q$ can be
realized.

Let $k$ be number field, $A$ a quaternion algebra over $k$,
$\O$ an  order in $A$, and $\Gamma$ a torsion free subgroup of
finite index in $\O^1$. Then, as described in subsection
\ref{subsec:arith}, each complex embedding of $k$ induces a map
$\Gamma\to \PSL(2,\C)$, each real embedding at which $A$ is unramified
induces a map $\Gamma\to \PSL(2,\R)$ and, via these maps, $\Gamma$
acts discretely with finite co-volume on the product
$$\X:=\prod_{i=1}^{r_2}\H^3\times\prod_{j=1}^{r_1^{u}}\H^2$$ of copies
of $\H^3$ and $\H^2$ (here $r_1^{u}$ is the number of real places of
$k$ at which $A$ is unramified). Denote $Y=\X/\Gamma$. Each projection
of $\X$ to one of the $\H^3$ factors gives a codimension $3$ foliation
on $\X$ which is preserved by the $\Gamma$--action, so $Y$ inherits a
codimension $3$ foliation from each of these projections. This is a
transversally hyperbolic foliation: there is a metric on the normal
bundle of the foliation which induces a hyperbolic metric on any local
transverse section.  Similarly, each projection to $\H^2$ gives a
codimension $2$ transversally hyperbolic foliation.

Now assume that $k$ is a concrete non-real number field, i.e., it
comes with a particular complex embedding singled out (which we call
the \emph{concrete embedding}). Pick the corresponding codimension $3$
foliation $\F$. Let $M^3\to Y$ be an immersion of a $3$-manifold to
$Y$ which is everywhere transverse to $\F$. So $M^3$ has an induced
hyperbolic metric. If $M^3$ is compact this metric is, of course,
complete of finite volume. We are interested also in the case that
$M^3$ is not compact, but we require then that the metric be complete
of finite volume (as we will see, this can only happen if $A$ is
 unramified over $k$).

\begin{theorem}
  The invariant trace field and quaternion algebra for $M^3$ embed in
  $k$ resp.\ $A$ (as concrete field and quaternion algebra). Moreover,
  $M$ has integral traces.

  Conversely, up to commensurability, every finite volume hyperbolic
  $3$--manifold with invariant quaternion algebra in $A$ (and hence
  invariant trace field in $k$) and with integral traces occurs this
  way.
\end{theorem}
\begin{proof} Suppose first that $M\hookrightarrow Y$ is as described
  in the theorem. Then $M$ inherits a hyperbolic metric locally from
  the metric transverse to the foliation $\F$. Consider one component
  $\tilde M\to \X$ of the pullback to the universal cover $\X$ of
  $Y$. By assumption the projection $\X\to\H^3$ to the first factor
  restricts to a proper local isometry, hence an isometry, of $\tilde
  M$ to $\H^3$. It follows that $M=\tilde M/\Gamma_0=\H^3/\Gamma_0$,
  where $\Gamma_0$ is a subgroup of the group $\Gamma$. The invariant
  trace field and invariant quaternion algebra of $\Gamma_0$ therefore
  embed in the invariant trace field $k$ and invariant quaternion
  algebra $A$ of $\Gamma$.

  Conversely, suppose $M$ has invariant trace field in $k$ and
  invariant quaternion algebra in $A$ and has integral traces. By
  going to a finite cover if necessary, we can assume the trace field
  of $M$ is in $k$. Write $M=\H^3/\Delta$ with $\Delta\subset
  \PSL(2,\C)$ and let $\overline \Delta$ be the inverse image of
  $\Delta$ in $\SL(2,\C)$. Then $A$ is the $k$--subalgebra of
  $M_2(\C)$ generated by $\overline\Delta$. The subring $\O\subset A$
  consisting of $\O_k$--linear combinations of elements of $\overline
  \Delta$ is an order in $A$, and $\overline \Delta\subset\O^1$ (see
  the proof of Theorem 8.3.2 in \cite{mr}).  Now any two orders in $A$
  are commensurable, so by going to a finite cover of $M$ if necessary
  we can assume that the order $\O$ we have here is contained in the
  order we used to construct $Y$, and therefore that $\Delta\subset
  \Gamma$. 


The developing map $\tilde M\to \H^3$  is
$\Delta$--equivariant for the action of $\Delta=\pi_1(M)$ by covering
transformations on $\tilde M $ and the given action of $\Delta$ on
$\H^3$. The latter is induced from the inclusion of $\bar\Delta$ in
$\O^1\subset SL(2,\C)$ coming from the concrete embedding $k\to \C$.

 Now the non-concrete complex embeddings of $k$ give actions of
$\Delta=\pi_1(M)$ on $\H^3$ which are not discrete. But, by Lemma
\ref{equivariant} below, 
for each of these we can construct a smooth $\Delta$--equivariant map
$\tilde M\to \H^3$. Similarly, we construct smooth equivariant maps
$\tilde M\to\H^2$ for each unramified real embedding. Together
these maps give a $\Delta$-equivariant map of $\tilde M$ to
$\X=\prod_{i=1}^{r_2}\H^3\times\prod_{j=1}^{r_1^{u}}\H^2$, where
$\Delta$ acts on $\X$ as a subgroup of $\Gamma$.  We thus get an
induced map of $\tilde M/\Delta=M$ to $\X/\Gamma=Y$, which clearly
does what is required.
\end{proof}
We used the following well known lemma:
\begin{lemma}\label{equivariant}
  If $X$ is a simplicial or CW-complex then for any action of
  $\pi_1(X)$ on a contractible space $Y$ there is a
  $\pi_1(X)$--equivariant map of $\tilde X$ to $Y$, and it is unique
  up to equivariant homotopy. Moreover, if $X$ and $Y$ are smooth
  manifolds and the action of $\pi_1(X)$ on $Y$ is by diffeomorphisms
  then this map can be chosen to be smooth.
\end{lemma}
\begin{proof}
  Indeed, one constructs the map inductively over skeleta of $\tilde
X$. If $X$ is smooth one can triangulate $X$ and construct the smooth
map inductively over thin neighborhoods of the skeleta. At the $k$-th
step one chooses a lift of each $k$--simplex and first extends the
smooth map already defined on a neighborhood of the boundary of this
lifted $k$--simplex smoothly to a neighborhood of the whole
$k$--simplex and then defines the map on $\pi_1(X)$--images of this
neighborhood by equivariance.
\end{proof}

One can extend the theorem to remove the restriction that the
hyperbolic manifold have integral traces. For each prime $\frak p
\subset \O_k$ which one wishes to allow in denominators of traces one
should add the corresponding Bass-Serre tree for $\frak p$ (see, e.g.,
\cite[Chapter VI]{smith conjecture}) as a factor on the right side of
the product of factors defining $\X$.  Then $Y$ is no longer a
manifold, but the foliation is still defined and any transversal to it
will be a manifold.

The theorem applies also to hyperbolic surfaces. Suppose $k$ has a
chosen real place (and is no longer required to have at least one
complex place) and $A$ is now unramified at this chosen real place.
We consider the foliation of $Y$ given by projecting $\X$ to the
corresponding $\H^2$ factor. An immersion $M^2\to Y$ transversal to
this foliation will induce a hyperbolic structure on $M^2$ with
invariant quaternion algebra in the concrete quaternion algebra $A$,
and again, any finite volume hyperbolic surface with integral traces
and invariant quaternion algebra in $A$ occurs this way up to
commensurability.  As before, the integral trace restriction can be
avoided by allowing also some Bass-Serre tree factors in $\X$.

In the $2$-dimensional case the existence of surfaces with given
invariant trace field can often be shown. For example, explicit
computation for the character varieties of ``small'' surfaces point to
the lack of restriction on what fields occur. And in at least one case
the existence of the transversals in $Y$ has been shown directly: if
$k=\Q(\sqrt d)$ is a real quadratic field and $\Gamma=\PSL(2,\O(\sqrt
d))$ (so $A$ is totally unramified) then $Y=\H^2\times\H^2/\Gamma$ is
a Hilbert modular surface and Hirzebruch and Zagier \cite{hz}
constructed many Riemann surfaces in $Y$ which are transverse to both
foliations. In this case these surfaces are all arithmetic, so they
give nothing new (they have invariant trace field $\Q$ and quaternion
algebra ramified at a possibly empty set of finite places, but
becoming unramified on extending scalars to $\Q(\sqrt d)$).

In the $3$-dimensional case the best evidence that such transversals
might always exist may be the richness of the collection of fields and
quaternion algebras provided by the 3-manifold census and snap, plus
the fact that their existence seems very likely in the 2-dimensional
case.

There is a $3$--dimensional foliation of $Y$ transverse to $\mathcal
F$, provided by the projection of $\X$ to the factors other than the
one used to construct $\mathcal F$. W. Thurston (private
communication) has suggested that one might seek immersed 3-manifolds which
are everywhere almost tangent to this foliation (and hence transverse
to $\mathcal F$). This would be very
interesting from the point of view of realizing Bloch invariants,
since it would realize Bloch invariants for which the components
other than the ``concrete component'' (giving $\vol(M)$) of the Borel
regulator are small with respect to volume.


\begin{thebibliography}{99}
\bibitem{borel} A. Borel, Commensurability classes and volumes of
  hyperbolic 3-manifolds.  Ann. Scuola Norm. Sup. Pisa {\bf 8} (1981),
  1--33.
\bibitem{census} P.J. Callahan, M.V. Hildebrand and J.R. Weeks,
     A census of cusped hyperbolic 3-manifolds,
     Mathematics of Computation {\bf68} (1999) 321--332.
\bibitem{cheeger-simons} J.~Cheeger and J.~Simons, \newblock
  Differential characters and geometric invariants, \newblock Springer
  Lect. Notes in Math. {\bf1167} (1985), 50--80.
\bibitem{chern-simons} S.~Chern, J.~Simons, \newblock {Some cohomology
    classes in principal fiber bundles and their application to
    Riemannian geometry}, \newblock
  Proc. Nat. Acad. Sci. U.S.A. {\bf68} (1971), 791--794.
\bibitem{cghn} D. Coulson, O. Goodman, C. Hodgson, W.D. Neumann,
  Computing arithmetic invariants of 3-manifolds, Experimental
  Mathematics {\bf9} (2000), 127--152
\bibitem{SnapPy} M. Culler, N. Dunfield, SnapPy (a python user
  interface to the SnapPea kernel), \url{http://www.math.uic.edu/~t3m/SnapPy}
\bibitem{dupont-sah} Johan L.~Dupont, Chin Han Sah, Scissors
  congruences II, J.\ Pure and App.\ Algebra {\bf25} (1982), 159--195.
\bibitem{zickert1} Johan L. Dupont,  Christian K. Zickert. A
  dilogarithmic formula for the Cheeger-Chern-Simons class.
  Geom. Topol.  {\bf10} (2006), 1347--1372
\bibitem{goette-zickert} Sebastian Goette, Christian K. Zickert. The
  extended Bloch group and the Cheeger-Chern-Simons class.
  Geom. Topol.  {\bf11} (2007), 1623--1635.
\bibitem{hz} F. Hirzebruch and D. B. Zagier, Intersection numbers of
  curves on hilbert modular surfaces and modular forms of nebentypus,
  Invent. math. {\bf 36} (1976), 57--113.
\bibitem{ruberman} Robert Meyerhoff, Daniel Ruberman.  Mutation and
  the $\eta$-invariant.  J. Differential Geom. 31 (1990), no. 1,
  101--130
\bibitem{pitsch} Michel Massey, Wolfgang Pitsch, J\'er\^ome Scherer,
  Generalized orientations and the Bloch invariant, J. K-Theory,
  Available on CJO 18 Nov 2009  doi:10.1017/is009009019jkt093 
\bibitem{meyerhoff-ruberman} Robert Meyerhoff, Daniel Ruberman.
  Cutting and pasting and the $\eta$-invariant.  Duke Math. J. {\bf61}
  (1990), 747--761
\bibitem{mr}Colin Maclachlan, Alan W. Reid. \emph{The arithmetic of
    hyperbolic 3-manifolds.} Graduate Texts in Mathematics,
  {\bf219}. Springer-Verlag, New York, 2003
\bibitem{m-suslin}A.S.~Merkurʹev, A.A.~Suslin, The group $K_3$ for a
  field. (Russian) Izv. Akad. Nauk SSSR Ser. Mat.  {\bf54} (1990),
  522--545; translation in Math. USSR-Izv.  {\bf36} (1991), 541--565.
\bibitem{smith conjecture} Morgan,~J.~W. and Bass,~H, \emph{eds},
  \emph{The Smith Conjecture}, (Academic Press 1984).
\bibitem{neumann-hilbert} Walter~D.~Neumann, \newblock Hilbert's 3rd
  problem and 3-manifolds, {\it The Epstein Birthday Schrift} (Igor
  Rivin, Colin Rourke and Caroline Series, editors), Geometry and
  Topology Monographs Volume 1 (International Press 1998 and
  www.maths.warwick.ac.uk/gt), 383--411
\bibitem{extbloch} Walter D. Neumann, Extended Bloch group and the
  Cheeger-Chern-Simons class, Geom. Topol. {\bf 8} (2004), 413--474
\bibitem{nr} Walter~D.~Neumann and Alan~W.~Reid, \newblock {Arithmetic of
    hyperbolic manifolds}, \newblock In {\em Topology 90, Proceedings
    of the Research Semester in Low Dimensional Topology at Ohio
    State}, Walter de Gruyter Verlag, Berlin - New York, 1992,
  273--310.
\bibitem{neumann-yang1} Walter~D.~Neumann and Jun~Yang, \newblock Problems
  for $K$-theory and Chern-Simons invariants of hyperbolic
  3-manifolds, \newblock L'Enseignement Math\-\'em\-at\-ique {\bf41}
  (1995), 281--296.
\bibitem{neumann-yang2} Walter~D.~Neumann and Jun~Yang, \newblock
  Invariants from triangulation for hyperbolic 3-manifolds, \newblock
  Electronic Research Announcements of the Amer. Math. Soc. {\bf1} (2)
  (1995), 72--79.
\bibitem{neumann-yang3} Walter~D.~Neumann and Jun~Yang, \newblock Bloch
  invariants of hyperbolic 3-manifolds, \newblock Duke
  Math. J. {\bf96} (1999), 29-59
\bibitem{reid} Alan~W.~Reid, \newblock A note on trace-fields of
  {Kleinian} groups, \newblock { Bull. London Math. Soc.} {\bf 22}
  (1990), 349--352.
\bibitem{suslin}A.A.~Suslin, $K_3$ of a field, and the Bloch
  group. (Russian) Translated in Proc. Steklov Inst. Math. {\bf4}
  (1991) 217--239. Galois theory, rings, algebraic groups and their
  applications (Russian).  Trudy Mat. Inst. Steklov.  {\bf 183}
  (1990), 180--199, 229.
\bibitem{vigneras} M-F Vign\'eras, Arithm\'etique des alg\`ebres de
  Quaternions. L.N.M.  800, Springer-Verlag, 1980.
\bibitem{snappea} J. Weeks, Snappea, the program,
  \url{http://www.geometrygames.org/SnapPea/index.html}.
\bibitem{zickert2}Christian K Zickert, The Volume and 
Chern-Simons
  invariant of a representation, Duke Math. J. {\bf 150} (2009), 489--532.
\bibitem{zickert3}Christian K Zickert, The extended Bloch group and
  algebraic K-theory, arXiv:0910.4005
\end{thebibliography}
\end{document}